\documentclass[a4paper,12pt]{article}
\pdfoutput=1

\usepackage[utf8]{inputenc}
\usepackage{csquotes, scalerel, stackengine}
\usepackage[dvipsnames]{xcolor}
\usepackage{soul}
\usepackage{todonotes}
\usepackage[backend=biber,style=apa]{biblatex}
\usepackage{tikz}
\usetikzlibrary{graphs,graphs.standard,quotes}
\usepackage{amsfonts, amsmath, amsthm, amssymb, enumerate, fancyhdr, gensymb, lastpage, mathtools, cleveref}
\usepackage{breqn}

\let\cite\parencite

\stackMath
\newcommand\reallywidehat[1]{%
\savestack{\tmpbox}{\stretchto{%
  \scaleto{%
    \scalerel*[\widthof{\ensuremath{#1}}]{\kern.1pt\mathchar"0362\kern.1pt}%
    {\rule{0ex}{\textheight}}%
  }{\textheight}%
}{2.4ex}}%
\stackon[-6.9pt]{#1}{\tmpbox}%
}

\newtheorem{lemma}{Lemma}[section]
\newtheorem{proposition}[lemma]{Proposition}
\newtheorem{theorem}[lemma]{Theorem}
\newtheorem{corollary}[lemma]{Corollary}

\theoremstyle{definition}
\newtheorem{definition}[lemma]{Definition}

\newtheorem{example}[lemma]{Example}
\newtheorem{question}[lemma]{Question}

\newcommand\cW{\mathcal{W}}
\newcommand\cG{\mathcal{G}}
\newcommand\cA{\mathcal{A}}

\newcommand\R{\mathbb{R}}

\newcommand\C{\mathbb{C}}

\newcommand\Sym{\mathfrak{S}}

\newcommand\V{\mathrm{V}_{\Omega}}

\newcommand\Vq{\mathrm{V}_{q}}
\newcommand\Iq{\mathrm{I}_{q}}
\newcommand\IA{\mathrm{I}_{\cA}}

\newcommand\Rq{R_q}
\newcommand\Sq{S_q}
\newcommand\thetaq{\theta_q}

\makeatletter
    \def\ps@copyright{\ps@empty
    \def\@oddfoot{\hfil\small\copyright 2026, Madelyn Andersen}}
\makeatother

\title{Kernel zero-sets, quantum graph ideals, and Hadamard graphons}
\author{Madelyn Andersen}

\begin{document}

\maketitle

\begin{abstract}
Graphons are symmetric measurable functions that arise as limiting objects of dense graph sequences. A kernel variety is a zero-set cut out by homomorphism-density equations $t(g,W)=0$, where $g$ is a quantum graph. We study these zero-sets using the complex algebra of unlabeled quantum graphs and the normalized homomorphism-density polynomial map to the invariant ring of a finite weighted target graph. For any fixed ambient class of bounded kernels, the quantum graphs that vanish on a given subset form an ideal. This gives a Zariski-type topology on kernel zero-sets and a Nullstellensatz-style statement after passing through the invariant polynomial ring. We also compute the corresponding ideals for Hadamard graphons using an explicit probability formula.
\end{abstract}


\section{Introduction}\label{section: Intro}

Graphons naturally arise as limiting objects for convergent sequences of graphs \cite{glasscock2015graphon, glasscock2016graphon, hombook, BORGS20081801}. There have been several studies examining the theoretical properties of graphons in recent years. A natural line of inquiry is when sequences of graphs converge trivially and when they converge to a non-zero limiting object; this consideration was studied by Borgs, Chayes, Cohn, and Holden in \cite{borgs2018sparse}. Graphons may be considered as symmetric measurable functions on probability spaces \cite{hombook, borgs2010left}. From this probabilistic perspective, they may also be used to study $\sigma$-finite measure spaces and compact metric spaces; for more details we refer the reader to \cite{janson2016graphons, glasscock2015graphon, borgs2010left, borgs2018sparse}.

Graphons also appear in a variety of applications. Graphons are used in network classification and training algorithms; see \cite{sabanayagam2021graphon, hu2021training, maskey2021transferability, xu2021learning, ruiz2020graphon, ruiz2021graph}. A graphon Fourier transform has applications to signal processing; see \cite{2021, ghandehari2021noncommutative}. Graphon mean field games are used to model interacting particle systems; see \cite{bayraktar2021graphon, bayraktar2021graphons, gao2021lqg, aurell2021stochastic, bassino2021random, gao2021, athreya2019graphonvalued, vasal2020sequential}.

In comparison to the probabilistic study and algorithmic applications of graphons, the algebraic study of graphons is still relatively unexplored; some discussion of the current state of the algebraic study is in \cite{hombook}. The semigroup structure of the set of all graphons is discussed in \cite{2021nagy}. In this paper, we further address the algebraic study of graphons by connecting graphons to familiar objects from algebraic geometry, namely zero loci, ideals, and polynomial varieties \cite{hartshorne}.

We begin by introducing kernels, graphons, homomorphism density, and quantum graphs in \cref{section: BG}. We then define the normalized homomorphism-density polynomial representation in \cref{section: Cat}. The normalization is important: homomorphism density satisfies $t(K_0,W)=t(K_1,W)=1$ for every kernel $W$, so the polynomial representation used for kernel zero-sets should send $K_0-K_1$ to $0$. In \cref{section: Var and Ideals}, we introduce kernel zero-sets and their vanishing ideals. In \cref{section: Res}, we prove that these vanishing sets define a Zariski-type topology and prove a Nullstellensatz-style result through the invariant polynomial ring. Finally, in \cref{section: Had}, we apply this framework to Hadamard graphons.

\section{Background}\label{section: BG}

Although graphons originally arose as limiting objects for convergent sequences of graphs \cite{hombook, glasscock2015graphon, glasscock2016graphon, BORGS20081801}, we use the following analytic definition.

\begin{definition}\label{definition:kernel}[\cite{hombook}]
Let $\cW$ be the space of all bounded symmetric measurable functions $W:[0,1]^2\to\R$. We call the elements of $\cW$ \textit{kernels}. Let $\cW_0$ denote the set of all kernels $W\in\cW$ such that $0\leq W\leq 1$. We call the elements of $\cW_0$ \textit{graphons}.
\end{definition}

Graphons only exist non-trivially as limits of sufficiently dense graph sequences. For a deeper understanding of what it means for a graph to be sufficiently dense, we refer the reader to \cite{BORGS20081801, borgs2018sparse}.

In order to use an algebraic approach to study graphons, we introduce homomorphism density. When considering two finite graphs, the homomorphism density is the probability that a randomly chosen map between the vertex sets preserves edge adjacency; see \cite{hombook, BORGS20081801}. We extend this notion to kernels and loopless multigraphs as follows.

\begin{definition}\label{definition:hom density}
For every $W\in\cW$ and every loopless multigraph $F=(V,E)$, define the homomorphism density
\begin{align*}
   t(F,W)=\int_{[0,1]^V}\prod_{ij\in E}W(x_i,x_j)\prod_{i\in V}dx_i.
\end{align*}
\end{definition}

We further generalize multigraphs into formal linear combinations.

\begin{definition}\label{definition:quantum graph}
A \textit{quantum graph} is a formal linear combination of finitely many loopless unweighted multigraphs with complex coefficients. Let
\[
\cA:=\C\cG_0
\]
denote the complex vector space spanned by finite unlabeled loopless multigraphs. The multigraphs that occur with nonzero coefficient are called the constituents of the quantum graph. Thus every $g\in\cA$ can be written as
\[
    g=\sum_{i=1}^n\alpha_iF_i,
\]
where $\alpha_i\in\C$ and each $F_i$ is a loopless multigraph.
\end{definition}

We use lowercase letters such as $g$ and $h$ for quantum graphs and capital letters such as $F$ and $G$ for ordinary multigraphs. Homomorphism density extends linearly to quantum graphs:
\[
    t\!\left(\sum_i\alpha_iF_i,W\right)=\sum_i\alpha_i t(F_i,W).
\]
The product on $\cA$ is induced by disjoint union:
\[
    FG:=F\sqcup G,
\]
and then extended bilinearly. With this multiplication, $\cA$ is a commutative $\C$-algebra whose unit is the empty graph $K_0$.

\section{The polynomial representation}\label{section: Cat}

We now recall the polynomial representation method for quantum graphs. This method is based on the homomorphism-polynomial construction of \cite{SCHRIJVER2009502, hombook}, but we use a normalized version because the zero-sets in this paper are defined by homomorphism \emph{densities}.

Fix $q\geq1$. Let
\[
    \Sq:=\C[x_{ij}:1\leq i\leq j\leq q]
\]
be the full polynomial ring in the entries of a symmetric $q\times q$ matrix $X=(x_{ij})$. The symmetric group $\Sym_q$ acts on $\Sq$ by simultaneously permuting the indices. Let
\[
    \Rq:=\Sq^{\Sym_q}
\]
be the invariant ring.

\begin{definition}\label{definition: theta q}
For a finite loopless multigraph $F$, define
\[
\thetaq(F)
=
q^{-|V(F)|}
\sum_{\varphi:V(F)\to[q]}
\prod_{uv\in E(F)}x_{\varphi(u)\varphi(v)}.
\]
Extend $\thetaq$ linearly to a map
\[
\thetaq:\cA\to\Rq.
\]
We call $\thetaq(g)$ the \textit{normalized homomorphism-density polynomial} of $g$ with respect to $q$.
\end{definition}

If $W_X$ is the equal-$q$-step kernel whose block value on the $(i,j)$ block is $x_{ij}$, then
\[
    t(F,W_X)=\thetaq(F)(X).
\]
This is why the normalized map is the appropriate polynomial representation for homomorphism-density zero-sets. For example,
\[
    \thetaq(K_0)=1,\qquad \thetaq(K_1)=1,
\]
so $\thetaq(K_0-K_1)=0$, matching the identity $t(K_0-K_1,W)=0$ for every kernel $W$. This is why the unnormalized homomorphism-number polynomial is not used as the main ideal map in this paper: it sends $K_0-K_1$ to the nonzero constant $1-q$ when $q>1$.

\begin{proposition}\label{proposition: theta homomorphism}
For every fixed $q\geq1$, the map
\[
    \thetaq:\cA\to\Rq
\]
is a surjective $\C$-algebra homomorphism. Hence
\[
    \cA/\ker\thetaq\cong \Rq.
\]
\end{proposition}

\begin{proof}
For ordinary multigraphs $F$ and $G$, the product in $\cA$ is disjoint union. Therefore
\begin{align*}
\thetaq(FG)
&=
q^{-|V(F)|-|V(G)|}
\sum_{\varphi:V(F)\sqcup V(G)\to[q]}
\prod_{uv\in E(F)\cup E(G)}
 x_{\varphi(u)\varphi(v)}\\
&=\left(
q^{-|V(F)|}
\sum_{\varphi_F:V(F)\to[q]}
\prod_{uv\in E(F)}x_{\varphi_F(u)\varphi_F(v)}
\right)
\left(
q^{-|V(G)|}
\sum_{\varphi_G:V(G)\to[q]}
\prod_{uv\in E(G)}x_{\varphi_G(u)\varphi_G(v)}
\right)\\
&=\thetaq(F)\thetaq(G).
\end{align*}
Linearity gives multiplicativity for all quantum graphs $g,h\in\cA$, and the empty graph maps to $1$. Thus $\thetaq$ is a $\C$-algebra homomorphism.

Surjectivity follows from the Lov\'asz--Schrijver representation theorem for invariant homomorphism polynomials: the unnormalized homomorphism polynomials of graphs span $\Rq$ \cite{SCHRIJVER2009502, hombook}. Multiplying the polynomial corresponding to a graph $F$ by the nonzero scalar $q^{-|V(F)|}$ does not change the linear span. Hence the normalized image also spans $\Rq$.
\end{proof}

The map $\thetaq$ is not asserted to be an isomorphism or a ring-theoretic retraction. The quotient description $\cA/\ker\thetaq\cong\Rq$ is the correct algebraic relationship needed below.

\section{Varieties and Ideals}\label{section: Var and Ideals}

We now introduce the zero-set and ideal notation used throughout the rest of the paper. Fix once and for all an ambient class of kernels
\[
    \Omega\subseteq\cW.
\]
One may take $\Omega=\cW$, $\Omega=\cW_0$, the set of finite-rank kernels, or another chosen class. The ideal and topology statements below hold for any fixed choice of $\Omega$.

\begin{definition}\label{definition: kernel zero set}
For a subset $Q\subseteq\cA$, define the kernel zero-set
\[
    \V(Q)=\{W\in\Omega:t(g,W)=0\text{ for every }g\in Q\}.
\]
If $Q=\{g\}$, we write $\V(g)$ instead of $\V(\{g\})$. We call a subset $Y\subseteq\Omega$ \textit{algebraic} if $Y=\V(Q)$ for some $Q\subseteq\cA$.
\end{definition}

A graphon variety is obtained by taking $\Omega=\cW_0$ or by intersecting a kernel zero-set with $\cW_0$. We allow systems of equations by allowing $Q$ to have more than one element.

For reference, we also recall the usual notion of finite rank.

\begin{definition}\label{definition:finite rank}[\cite{hombook}]
For every $W\in\cW$, define the operator $T_W:L_1[0,1]\to L_\infty[0,1]$ by
\[
    (T_Wf)(x)=\int_0^1 W(x,y)f(y)\,dy.
\]
If this operator has finite-dimensional range, then we call $W$ a finite-rank kernel.
\end{definition}

\begin{definition}\label{definition: ideals}
For $Y\subseteq\Omega$, define the quantum-graph vanishing ideal
\[
    \IA(Y)=\{g\in\cA:t(g,W)=0\text{ for every }W\in Y\}.
\]
For fixed $q$, define the polynomial image ideal
\[
    \Iq(Y)=\thetaq(\IA(Y))\subseteq\Rq.
\]
For an ideal $\mathfrak a\subseteq\Rq$, define
\[
    \Vq(\mathfrak a):=\V(\thetaq^{-1}(\mathfrak a)).
\]
\end{definition}

\section{Results}\label{section: Res}

The main ideal statement rests on the multiplicativity of homomorphism density under disjoint union. This does not require a finite-rank hypothesis.

\begin{lemma}\label{lemma: multiplicativity}
For all $g,h\in\cA$ and all $W\in\cW$,
\[
    t(gh,W)=t(g,W)t(h,W).
\]
\end{lemma}

\begin{proof}
It is enough to prove the statement for ordinary multigraphs $F$ and $G$ and then extend bilinearly. Since $FG=F\sqcup G$,
\begin{align*}
 t(FG,W)
&=\int_{[0,1]^{V(F)\sqcup V(G)}}
\prod_{uv\in E(F)}W(x_u,x_v)
\prod_{uv\in E(G)}W(y_u,y_v)\,dx\,dy\\
&=\left(
\int_{[0,1]^{V(F)}}
\prod_{uv\in E(F)}W(x_u,x_v)\,dx
\right)
\left(
\int_{[0,1]^{V(G)}}
\prod_{uv\in E(G)}W(y_u,y_v)\,dy
\right)\\
&=t(F,W)t(G,W).
\end{align*}
The general case follows by bilinearity.
\end{proof}

\begin{theorem}\label{theorem: vanishing ideals are ideals}
For every $Y\subseteq\Omega$, the set $\IA(Y)$ is an ideal of $\cA$. Consequently, $\Iq(Y)$ is an ideal of $\Rq$.
\end{theorem}

\begin{proof}
Let $g\in\IA(Y)$ and $h\in\cA$. For every $W\in Y$, \cref{lemma: multiplicativity} gives
\[
    t(gh,W)=t(g,W)t(h,W)=0.
\]
Thus $gh\in\IA(Y)$, so $\IA(Y)$ is an ideal of $\cA$.

By \cref{proposition: theta homomorphism}, $\thetaq:\cA\to\Rq$ is a surjective algebra homomorphism. The image of an ideal under a surjective ring homomorphism is an ideal, so $\Iq(Y)=\thetaq(\IA(Y))$ is an ideal of $\Rq$.
\end{proof}

\begin{corollary}\label{cor: hom ideal}
Let
\[
    \Omega_m:=\Omega\cap\{W\in\cW: W\text{ has rank at most }m\}.
\]
Then $\IA(\Omega_m)$ and $\Iq(\Omega_m)$ are ideals.
\end{corollary}

\begin{proof}
This is the special case $Y=\Omega_m$ of \cref{theorem: vanishing ideals are ideals}.
\end{proof}

\begin{proposition}\label{proposition: zariski topology}
The algebraic subsets of $\Omega$ are closed under finite unions and arbitrary intersections. Moreover, $\varnothing$ and $\Omega$ are algebraic. Hence the complements of algebraic subsets define a topology on $\Omega$, called the Zariski topology associated to homomorphism-density zero-sets.
\end{proposition}

\begin{proof}
Let $Y_1=\V(Q_1)$ and $Y_2=\V(Q_2)$. Define
\[
    Q_1Q_2:=\{gh:g\in Q_1,\ h\in Q_2\}.
\]
We claim that
\[
    Y_1\cup Y_2=\V(Q_1Q_2).
\]
If $W\in Y_1\cup Y_2$, then either all elements of $Q_1$ vanish at $W$, or all elements of $Q_2$ vanish at $W$. Hence every product $gh$ with $g\in Q_1$ and $h\in Q_2$ vanishes at $W$.

Conversely, suppose $W\in\V(Q_1Q_2)$ and $W\notin Y_1$. Then there exists $g\in Q_1$ such that $t(g,W)\neq0$. For every $h\in Q_2$,
\[
    0=t(gh,W)=t(g,W)t(h,W),
\]
so $t(h,W)=0$. Hence $W\in Y_2$. This proves the finite-union claim.

For arbitrary intersections,
\[
    \bigcap_{\alpha}\V(Q_\alpha)=\V\!\left(\bigcup_{\alpha}Q_\alpha\right).
\]
Finally,
\[
    \varnothing=\V(\{K_0\})=\V(\{K_1\}),
\]
because $t(K_0,W)=t(K_1,W)=1$ for every kernel $W$, while
\[
    \Omega=\V(\varnothing)=\V(\{0\}).
\]
Thus algebraic sets satisfy the axioms for closed sets of a topology.
\end{proof}

\begin{definition}\label{definition: irreducible}
A nonempty subset $Y\subseteq\Omega$ is \textit{irreducible} if it cannot be expressed as the union $Y_1\cup Y_2$ of two proper subsets, each of which is closed in $Y$ in the relative Zariski topology. The empty set is not considered irreducible.
\end{definition}

\begin{proposition}\label{proposition: ideal variety properties}
The following properties hold.
\begin{enumerate}[(a)]
    \item If $A_1\subseteq A_2\subseteq\cA$, then
    \[
        \V(A_1)\supseteq\V(A_2).
    \]
    \item If $Y_1\subseteq Y_2\subseteq\Omega$, then
    \[
        \IA(Y_1)\supseteq\IA(Y_2),\qquad \Iq(Y_1)\supseteq\Iq(Y_2).
    \]
    \item For arbitrary $Y_1,Y_2\subseteq\Omega$,
    \[
        \IA(Y_1\cup Y_2)=\IA(Y_1)\cap\IA(Y_2).
    \]
    For polynomial image ideals, the automatic statement is
    \[
        \Iq(Y_1\cup Y_2)\subseteq\Iq(Y_1)\cap\Iq(Y_2).
    \]
    \item For every $Y\subseteq\Omega$,
    \[
        \V(\IA(Y))=\overline{Y},
    \]
    where the closure is taken in the Zariski topology of \cref{proposition: zariski topology}.
\end{enumerate}
\end{proposition}

\begin{proof}
Parts (a), (b), and the equality in (c) follow directly from the definitions. The inclusion for $\Iq$ in (c) follows by applying $\thetaq$ to the equality for $\IA$.

It remains to prove (d). First,
\[
    Y\subseteq\V(\IA(Y)),
\]
because every element of $\IA(Y)$ vanishes on every point of $Y$. Since $\V(\IA(Y))$ is closed,
\[
    \overline Y\subseteq\V(\IA(Y)).
\]
Conversely, let $C$ be any closed subset of $\Omega$ containing $Y$. Write $C=\V(Q)$ for some $Q\subseteq\cA$. Since $Y\subseteq C$, every element of $Q$ vanishes on $Y$, so $Q\subseteq\IA(Y)$. By part (a),
\[
    \V(\IA(Y))\subseteq\V(Q)=C.
\]
This holds for every closed $C\supseteq Y$, so $\V(\IA(Y))\subseteq\overline Y$. Therefore $\V(\IA(Y))=\overline Y$.
\end{proof}

\begin{proposition}\label{proposition: radical inclusion}
For every ideal $\mathfrak a\subseteq\Rq$,
\[
    \sqrt{\mathfrak a}\subseteq \Iq(\Vq(\mathfrak a)).
\]
\end{proposition}

\begin{proof}
Let $f\in\sqrt{\mathfrak a}$. Since $\thetaq:\cA\to\Rq$ is surjective, choose $g\in\cA$ such that $\thetaq(g)=f$. Since $f\in\sqrt{\mathfrak a}$, there exists $r>0$ such that $f^r\in\mathfrak a$. Because $\thetaq$ is a ring homomorphism,
\[
    \thetaq(g^r)=\thetaq(g)^r=f^r\in\mathfrak a.
\]
Thus $g^r\in\thetaq^{-1}(\mathfrak a)$.

If $W\in\Vq(\mathfrak a)=\V(\thetaq^{-1}(\mathfrak a))$, then
\[
    0=t(g^r,W)=t(g,W)^r,
\]
where the last equality is \cref{lemma: multiplicativity}. Therefore $t(g,W)=0$ for every $W\in\Vq(\mathfrak a)$, so $g\in\IA(\Vq(\mathfrak a))$. Applying $\thetaq$ gives
\[
    f=\thetaq(g)\in\Iq(\Vq(\mathfrak a)).
\]
Hence $\sqrt{\mathfrak a}\subseteq\Iq(\Vq(\mathfrak a))$.
\end{proof}

We restate Hilbert's Nullstellensatz in the form used below.

\begin{theorem}[Hilbert's Nullstellensatz \cite{hartshorne}]\label{theorem: classical HN}
Let $k$ be an algebraically closed field, let $\mathfrak a$ be an ideal in $A=k[x_1,\ldots,x_n]$, and let $f\in A$ vanish at all points of $Z(\mathfrak a)$. Then $f^r\in\mathfrak a$ for some integer $r>0$.
\end{theorem}

The next theorem is the invariant-ring form needed for homomorphism-density polynomials. The point is that an ideal of $\Rq$ should not be treated as an ideal of the full polynomial ring $\Sq$; instead, one extends the ideal to $\Sq$ and contracts back to $\Rq$.

\begin{theorem}\label{HNAK1}
Let $\mathfrak a\subseteq\Rq$ be an ideal. Let $f\in\Rq$ vanish on
\[
    Z(\mathfrak a)=\{X\in\C^{q(q+1)/2}:h(X)=0\text{ for every }h\in\mathfrak a\}.
\]
Then
\[
    f\in\sqrt{\mathfrak a}.
\]
Equivalently, $f^r\in\mathfrak a$ for some integer $r>0$.
\end{theorem}

\begin{proof}
Recall that $\Sq=\C[x_{ij}:1\leq i\leq j\leq q]$ and $\Rq=\Sq^{\Sym_q}$. The finite group $\Sym_q$ acts on $\Sq$, and because we are working over $\C$, the Reynolds operator
\[
    \rho(s)=\frac{1}{|\Sym_q|}\sum_{\sigma\in\Sym_q}\sigma(s)
\]
is an $\Rq$-linear projection $\Sq\to\Rq$ whose restriction to $\Rq$ is the identity. In particular, for every ideal $\mathfrak a\subseteq\Rq$,
\[
    \mathfrak a\Sq\cap\Rq=\mathfrak a.
\]
Indeed, if $u\in\mathfrak a\Sq\cap\Rq$ and $u=\sum_i a_i s_i$ with $a_i\in\mathfrak a$ and $s_i\in\Sq$, then
\[
    u=\rho(u)=\sum_i a_i\rho(s_i)\in\mathfrak a.
\]
The reverse inclusion is immediate.

The zero set $Z(\mathfrak a)$ is the zero set in $\C^{q(q+1)/2}$ of the extended ideal $\mathfrak a\Sq\subseteq\Sq$. Since $f$ vanishes on this zero set, the classical Nullstellensatz in the full polynomial ring $\Sq$ gives
\[
    f\in\sqrt{\mathfrak a\Sq}.
\]
Since $f\in\Rq$, there exists $r>0$ such that
\[
    f^r\in\mathfrak a\Sq\cap\Rq=\mathfrak a.
\]
Thus $f\in\sqrt{\mathfrak a}$.
\end{proof}

\begin{theorem}\label{HNAK2}
Let $\mathfrak a\subseteq\Rq$ be an ideal and let $g\in\cA$. Suppose $\thetaq(g)$ vanishes on $Z(\mathfrak a)$. Let
\[
    V:=\Vq(\mathfrak a)=\V(\thetaq^{-1}(\mathfrak a)).
\]
Then
\[
    g\in\IA(V).
\]
Equivalently,
\[
    \thetaq(g)\in\Iq(V).
\]
\end{theorem}

\begin{proof}
By \cref{HNAK1}, $\thetaq(g)\in\sqrt{\mathfrak a}$. Hence there exists $r>0$ such that
\[
    \thetaq(g)^r\in\mathfrak a.
\]
Since $\thetaq$ is a ring homomorphism,
\[
    \thetaq(g^r)=\thetaq(g)^r\in\mathfrak a.
\]
Thus $g^r\in\thetaq^{-1}(\mathfrak a)$. If $W\in V=\V(\thetaq^{-1}(\mathfrak a))$, then $t(g^r,W)=0$. By \cref{lemma: multiplicativity},
\[
    t(g^r,W)=t(g,W)^r.
\]
Therefore $t(g,W)=0$. Hence $g\in\IA(V)$, and applying $\thetaq$ gives $\thetaq(g)\in\Iq(V)$.
\end{proof}

\begin{corollary}\label{corollary: irreducible prime}
Let $V\subseteq\Omega$ be a nonempty algebraic kernel set. Then
\[
    V\text{ is irreducible}\quad\Longleftrightarrow\quad \IA(V)\text{ is a prime ideal of }\cA.
\]
\end{corollary}

\begin{proof}
Suppose $V$ is irreducible and $gh\in\IA(V)$. Then for every $W\in V$,
\[
    0=t(gh,W)=t(g,W)t(h,W).
\]
Therefore
\[
    V\subseteq\V(g)\cup\V(h).
\]
By irreducibility, $V\subseteq\V(g)$ or $V\subseteq\V(h)$. Hence $g\in\IA(V)$ or $h\in\IA(V)$, so $\IA(V)$ is prime.

Conversely, suppose $\IA(V)$ is prime and $V=V_1\cup V_2$, where $V_1$ and $V_2$ are closed in $V$. Write $V_i=V\cap\V(Q_i)$. If both $V_i$ are proper subsets of $V$, then choose $g_i\in Q_i$ such that $g_i\notin\IA(V)$. Each $g_i$ vanishes on $V_i$, so $g_1g_2$ vanishes on $V_1\cup V_2=V$. Thus $g_1g_2\in\IA(V)$, contradicting primeness. Hence one of $V_1,V_2$ equals $V$, and $V$ is irreducible.
\end{proof}

The criterion in \cref{corollary: irreducible prime} is stated using the quantum-graph ideal $\IA(V)$. A criterion using $\Iq(V)$ requires additional hypotheses, because $\thetaq$ need not be injective.

\section{Hadamard Matrices}\label{section: Had}

\begin{definition}[\cite{ferber2018number}]
A \textit{Hadamard matrix of order $n$} is an $n\times n$ matrix $H$ whose entries are $\{\pm1\}$ and whose rows are pairwise orthogonal, i.e. $HH^T=nI_n$. Hadamard matrices are named after Jacques Hadamard, who studied them in connection with the maximal determinant problem.
\end{definition}

Hadamard matrices are widely studied. While we acknowledge the broad impacts of their applications, we will not touch on this subject. For greater discussion of their applications and Hadamard's studies, see \cite{ferber2018number, lofano2021hadamard, banica2021complex, breen2020hadamard, li2020note}. Throughout this section, we use examples of $2\times2$ Hadamard matrices to illustrate the preceding concepts.

\begin{definition}\label{definition:stepfunction}
A kernel $W\in\cW$ is called a \textit{stepfunction} if:
\begin{enumerate}[(i)]
    \item there is a partition $\{T_1,\ldots,T_k\}$ of $[0,1]$ into measurable sets;
    \item $W$ is constant on every product set $T_i\times T_j$.
\end{enumerate}
We call the sets $T_i$ the \textit{steps} of $W$.
\end{definition}

If $P=(p_{ij})$ is a symmetric $n\times n$ matrix, we construct an equal-step kernel $W_P$ by setting
\[
    T_i=\left[\frac{i-1}{n},\frac{i}{n}\right)
\]
for $1\leq i<n$, and $T_n=[(n-1)/n,1]$. Then define $W_P(x,y)=p_{ij}$ for $x\in T_i$ and $y\in T_j$.

\begin{definition}\label{definition:hadamard graphons}[\cite[Example 14.35, p. 251]{hombook}]
A symmetric $n\times n$ Hadamard matrix $B$ gives rise to a stepfunction kernel $W_B$. Define
\[
    U_B=\frac{W_B+1}{2}.
\]
Then $U_B$ is a graphon, called the \textit{Hadamard graphon} associated to $B$.
\end{definition}

\begin{example}\label{example: standard H}
Let
\[
    H=\begin{pmatrix}1&1\\1&-1\end{pmatrix}.
\]
The corresponding stepfunction kernel $W_H:[0,1]^2\to\R$ is
\[
W_H(x,y)=
\begin{cases}
1, & x,y\in [0,\frac12),\\
1, & x\in [0,\frac12),\ y\in [\frac12,1],\\
1, & x\in [\frac12,1],\ y\in [0,\frac12),\\
-1, & x,y\in [\frac12,1].
\end{cases}
\]
The associated Hadamard graphon is
\[
U_H(x,y)=\frac{1+W_H(x,y)}{2}=
\begin{cases}
1, & x,y\in [0,\frac12),\\
1, & x\in [0,\frac12),\ y\in [\frac12,1],\\
1, & x\in [\frac12,1],\ y\in [0,\frac12),\\
0, & x,y\in [\frac12,1].
\end{cases}
\]
\end{example}

The set of all Hadamard graphons together with the constant $1/2$ graphon forms a simple variety by \cite{hombook}.

\begin{definition}\label{def: hadamard probability}
Let $B=(b_{ij})$ be a symmetric $n\times n$ Hadamard matrix and let $F$ be a loopless multigraph. Define
\[
P(B,F)=
\frac{1}{n^{|V(F)|}}
\#\left\{
\varphi:V(F)\to[n]:
 b_{\varphi(u)\varphi(v)}=1\text{ for every }uv\in E(F)
\right\}.
\]
For a quantum graph $g=\sum_i\alpha_iF_i$, define
\[
    P(B,g):=\sum_i\alpha_iP(B,F_i).
\]
\end{definition}

\begin{proposition}\label{proposition: hadamard formula}
For every symmetric Hadamard matrix $B$ and every quantum graph $g\in\cA$,
\[
    t(g,U_B)=P(B,g).
\]
\end{proposition}

\begin{proof}
It is enough to prove the statement for an ordinary multigraph $F$ and then extend linearly. On each block $T_i\times T_j$,
\[
U_B(x,y)=
\begin{cases}
1, & b_{ij}=1,\\
0, & b_{ij}=-1.
\end{cases}
\]
For a map $\varphi:V(F)\to[n]$, the corresponding block assignment contributes $1$ to
\[
    \prod_{uv\in E(F)}U_B(x_u,x_v)
\]
exactly when $b_{\varphi(u)\varphi(v)}=1$ for every edge $uv\in E(F)$; otherwise it contributes $0$. Each block assignment has measure $n^{-|V(F)|}$. Therefore
\[
    t(F,U_B)=P(B,F).
\]
Linearity gives $t(g,U_B)=P(B,g)$.
\end{proof}

\begin{example}\label{example: triangle hadamard}
Let
\[
B=\begin{pmatrix}
1&-1\\
-1&-1
\end{pmatrix}.
\]
This is a symmetric Hadamard matrix. The graphon $U_B$ equals $1$ only on the block $T_1\times T_1$, where $T_1=[0,1/2)$, and equals $0$ elsewhere. For the triangle $K_3$, a nonzero contribution occurs exactly when all three vertices land in $T_1$. Hence
\[
    t(K_3,U_B)=\mu(T_1)^3=\left(\frac12\right)^3=\frac18.
\]
Equivalently, $P(B,K_3)=1/8$.
\end{example}

\begin{proposition}\label{proposition: positive diagonal}
If a symmetric Hadamard matrix $B$ has a $+1$ entry on the diagonal, then
\[
    t(F,U_B)>0
\]
for every ordinary multigraph $F$. Consequently, for every nonzero scalar $\alpha\in\C$,
\[
    t(\alpha F,U_B)=\alpha t(F,U_B)\neq0.
\]
\end{proposition}

\begin{proof}
Suppose $b_{ii}=1$. In the integral defining $t(F,U_B)$, restrict to the region where every vertex variable lies in the step $T_i$. On this region every edge factor equals $1$, and the region has positive measure $n^{-|V(F)|}$. Since $U_B$ is $0/1$-valued, the rest of the integral is nonnegative. Hence $t(F,U_B)>0$.
\end{proof}

\begin{theorem}\label{theorem: hadamard ideal}
Let $B$ be a symmetric Hadamard matrix and $U_B$ its Hadamard graphon. Then
\[
    \IA(\{U_B\})=
    \{g\in\cA:P(B,g)=0\}.
\]
\end{theorem}

\begin{proof}
By \cref{proposition: hadamard formula}, $t(g,U_B)=P(B,g)$ for every $g\in\cA$. Hence
\[
    g\in\IA(\{U_B\})
\]
if and only if $t(g,U_B)=0$, if and only if $P(B,g)=0$.
\end{proof}

\section{Future Work}\label{section: Fut}

The ideal and topology results above do not require finite rank; they hold for any fixed ambient class $\Omega\subseteq\cW$. This suggests that a useful next direction is not to extend the ideal theorem beyond finite rank, but rather to compare how different choices of $\Omega$ change the resulting Zariski topology and vanishing ideals.

The finite-$q$ polynomial quotient also leaves natural questions. Since $\thetaq$ need not be injective, equality after applying $\thetaq$ does not necessarily imply equality of kernel zero-sets. For example, when $q=1$,
\[
    \theta_1(K_2\sqcup K_2)=x^2,
    \qquad
    \theta_1(P_3)=x^2,
\]
where $P_3$ is the path on three vertices. Thus
\[
    D:=K_2\sqcup K_2-P_3
\]
lies in $\ker\theta_1$. But for the rank-one kernel $W(x,y)=xy$,
\[
    t(K_2\sqcup K_2,W)=t(K_2,W)^2=\left(\frac14\right)^2=\frac1{16},
\]
whereas
\[
    t(P_3,W)=\int_{[0,1]^3}xy\cdot yz\,dx\,dy\,dz=\frac1{12}.
\]
Hence
\[
    t(D,W)=\frac1{16}-\frac1{12}=-\frac1{48}\neq0.
\]
So fixed-$q$ polynomial data alone cannot determine kernel zero-sets in general.

\begin{question}
Characterize when two quantum-graph sets $Q_1,Q_2\subseteq\cA$ define the same kernel zero-set
\[
    \V(Q_1)=\V(Q_2),
\]
and compare this equivalence relation with the finite-$q$ quotient relations induced by $\ker\thetaq$.
\end{question}

\begin{question}
For a fixed ambient class $\Omega$, what additional hypotheses on an algebraic kernel set $V\subseteq\Omega$ guarantee that the polynomial image ideal $\Iq(V)$ retains enough information to characterize irreducibility or other geometric properties of $V$?
\end{question}

\printbibliography

\end{document}